\documentclass[a4paper,10pt]{amsart}
\usepackage[T1]{fontenc}
\usepackage{lmodern}
\usepackage[all]{xy}
\usepackage{amssymb}
\usepackage{enumitem,nicefrac,mathtools}
\usepackage[british]{babel}

\usepackage{xcolor}

\newcommand*{\MRref}[2]{\linebreak[0] \href{http://www.ams.org/mathscinet-getitem?mr=#1}{MR \textbf{#1}}}
\newcommand*{\arxiv}[1]{\linebreak[0] \href{http://www.arxiv.org/abs/#1}{arXiv:#1}}
\usepackage[pdftitle={Homotopy-theoretic E-theory and n-order},
  pdfauthor={Rasmus Bentmann},
  pdfsubject={Mathematics; MSC 18G20, 19K35, 46L80}
]{hyperref}
\usepackage[lite]{amsrefs}
\usepackage{microtype}

\usepackage[foot]{amsaddr}

\usepackage{graphicx}
\usepackage{tikz}
\usetikzlibrary{matrix,arrows,calc}
\tikzset{cd/.style=matrix of math nodes,row sep=2em,column sep=2em, text height=1.5ex, text depth=0.5ex}
\tikzset{cdar/.style=->,auto}
\tikzset{mid/.style={anchor=mid}} 
\tikzset{dar/.style={double,double equal sign distance,-implies}}
\tikzset{narrowfill/.style={inner sep=1pt, fill=white}}

\usepackage{array}
\usepackage[labelformat=simple]{subfig}

\DeclareMathOperator{\Ho}{Ho}
\DeclareMathOperator{\Hom}{Hom}

\newcommand*{\KK}{\textup{KK}}
\newcommand*{\K}{\textup{K}}
\newcommand*{\E}{\textup{E}}
\newcommand*{\bu}{\textup{bu}}

\newcommand*{\Derived}{\textup{Der}}
\newcommand*{\Repring}{\textup{R}}

\newcommand*{\Rspc}{\mathbf{R}}
\newcommand*{\Kspc}{\mathbf{K}}
\newcommand*{\Aspc}{\mathbf{A}}
\newcommand*{\buspc}{\mathbf{bu}}

\newcommand*{\Cstarsep}{\mathfrak{C^*sep}}

\newcommand*{\Ecat}{\mathfrak{E}}
\newcommand*{\bucat}{\mathfrak{bu}}

\newcommand*{\Mod}[1]{\mathfrak{Mod}\bigl(#1\bigr)} 

\newcommand*{\C}{\mathbb{C}}

\newcommand*{\Z}{\mathbb{Z}}
\newcommand*{\N}{\mathbb{N}}

\newcommand*{\Repr}{\mathcal R}

\newcommand*{\nb}{\nobreakdash}  

\newcommand*{\Bootstrap}{{\mathcal B}}

\newcommand*{\SW}{{\mathcal{SW}}}

\newcommand*{\Model}{{\mathcal M}}

\newcommand*{\Cstar}{\texorpdfstring{$C^*$\nb-}{C*-}}
\newcommand*{\Star}{\texorpdfstring{$^*$\nb-}{*-}}

\newcommand*{\defeq}{\mathrel{\vcentcolon=}}

\newcommand*{\inOb}{\mathrel{\in\in}}

\theoremstyle{plain}

\numberwithin{equation}{section}

\theoremstyle{plain}
\newtheorem{theorem}[equation]{Theorem}
\newtheorem{lemma}[equation]{Lemma}

\newtheorem{proposition}[equation]{Proposition}

\theoremstyle{definition}
\newtheorem{definition}[equation]{Definition}

\theoremstyle{remark}
\newtheorem{remark}[equation]{Remark}

\title
{Homotopy-theoretic $\E$-theory and $n$-order}

\author[Rasmus Bentmann]{Rasmus Bentmann}

\address{Department of Mathematical Sciences\\University of
Copenhagen\\Universitetsparken~5\\2100 Copenhagen \O \\Denmark}
\email{bentmann@math.ku.dk}



\thanks{The author was supported by the Danish National Research Foundation through the Centre for Symmetry and Deformation (DNRF92) and by the Marie Curie Research Training Network EU-NCG}

\begin{document}

\begin{abstract}
The bootstrap category in $\E$\nb-theory for \Cstar{}algebras over a finite space is embedded into the homotopy category of certain diagrams of $\Kspc$\nb-module spectra. Therefore it has infinite $n$\nb-order for every $n\in\N$. The same holds for the bootstrap category in $G$\nb-equivariant $\E$\nb-theory for a compact group~$G$ and for the Spanier--Whitehead category in connective $\E$\nb-theory.
\end{abstract}

\maketitle

\section{Introduction}

Triangulated categories arise in various contexts such as algebraic geometry, representation theory and algebraic topology. This motivates their distinction into \emph{algebraic}, \emph{topological} (and non-algebraic), and \emph{exotic} (that is, non-topological) triangulated categories; see~\cite{Schwede:Top_tri_cats}. Every algebraic triangulated category is topological. The converse is false; topological triangulated categories may exhibit certain torsion phenomena which cannot occur in algebraic triangulated categories. The most well-known such phenomenon is the fact that the endomorphism ring of the mod-$2$ Moore spectrum is not annihilated by multiplication by~$2$.

In~\cites{Schwede:Alg_vs_top,Schwede:Top_tri_cats}, Stefan Schwede introduced the notion of \emph{$n$\nb-order} for triangulated categories (a non-negative integer or infinity for every $n\in\N$). This is an invariant (up to triangulated equivalence) that can often be used to distinguish non-algebraic triangulated categories from algebraic ones by measuring the occurence of the afore-mentioned torsion phenomena: Schwede shows that the $n$\nb-order of every algebraic triangulated category is infinite for every $n\in\N$; on the other hand, he proves that, if~$p$ is a prime number, the Spanier--Whitehead category in stable homotopy theory has $p$\nb-order equal to~$p-1$.

One aim of this note is to determine the $n$\nb-order of certain triangulated categories arising in \Cstar{}algebra theory. More specifically, we are interested in the bivariant homology theories
\begin{itemize}
\item
connective $\E$\nb-theory for separable \Cstar{}algebras, denoted by~$\bu$, as defined by Thom in~\cite{Thom:Thesis},
\item
$G$\nb-equivariant $\E$\nb-theory for separable \Cstar{}algebras with a continuous action of a compact group~$G$ by \Star{}automorphisms, denoted by~$\E^G$, as defined by Guentner, Higson and Trout in~\cite{Guentner-Higson-Trout:Equivariant_E},
\item
ideal-related $\E$\nb-theory for separable \Cstar{}algebras over a finite space~$X$, denoted by~$\E(X)$, as defined by Dadarlat and Meyer in~\cite{Dadarlat-Meyer:E_over_space}.
\end{itemize}

These give rise to triangulated categories denoted by~$\bucat$, $\Ecat^G$ and~$\Ecat(X)$, respectively. The \emph{Spanier--Whitehead category} $\SW_\bu\subset\bucat$ in connective $\E$\nb-theory is the thick triangulated subcategory of~$\bucat$ generated by the \Cstar{}algebra~$\C$ of complex numbers. The \emph{bootstrap categories} $\Bootstrap_\E^G\subset\Ecat^G$ and~$\Bootstrap_\E(X)\subset\Ecat(X)$ are the $\aleph_0$\nb-localizing subcategories generated by the objects with one-dimensional underlying \Cstar{}algebra, respectively. (While there are no non-trivial $G$\nb-actions by \Star{}automorphisms on~$\C$, there are as many mutually non-isomorphic ways to turn~$\C$ into a \Cstar{}algebra over~$X$ as there are points in the space~$X$).

Our computational result is the following; it may be regarded as a generalization of Schochet's observation in \cite{Schochet:Top4}*{Proposition~2.4}, stating that $\K$-theory with coefficients in $\Z/n$ is annihilated by multiplication by~$n$.

\begin{theorem}
  \label{thm:main}
The triangulated categories~$\SW_\bu$, $\Bootstrap_\E^G$ and~$\Bootstrap_\E(X)$ have infinite $n$\nb-order for every $n\in\N$.
\end{theorem}

The theorem is an application of the following result from \cite{Schwede:Top_tri_cats}*{Example~2.9}, which is based on results due to Tyler Lawson and to Vigleik Angeltveit~\cite{Angeltveit}.

\begin{theorem}
  \label{thm:lawson_angeltveit}
Let~$\Rspc$ be a commutative symmetric ring spectrum such that $\pi_*\Rspc$ is torsion-free and concentrated in even dimensions. Let~$\Aspc$ be an $\Rspc$\nb-algebra spectrum. Then the derived category of~$\Aspc$\nb-module spectra has infinite $n$\nb-order for every $n\in\N$.
\end{theorem}

In order to apply this theorem, we need to embed our bootstrap categories into appropriate derived categories of module spectra. Theorem~\ref{thm:main} then follows because the $n$\nb-order can only increase when we pass to a triangulated subcategory. For connective $\E$\nb-theory and $G$\nb-equivariant $\E$\nb-theory we get the desired embeddings essentially for free from the literature. More specifically, we use a result from Andreas Thom's thesis in the case of connective $\E$\nb-theory and a construction of Dell'Ambrogio--Emerson--Kandelaki--Meyer in the $G$\nb-equivariant case. In both cases, Theorem~\ref{thm:lawson_angeltveit} may be applied with $\Aspc=\Rspc$. This is not surprising because the categories~$\bucat$ and~$\Ecat^G$ are monoidal.

In the case of ideal-related $\E$\nb-theory, we have to work a little harder to obtain the desired embedding. We apply the proposition with~$\Rspc$ equal to the Dell'Ambrogio--Emerson--Kandelaki--Meyer spectrum $\Kspc=\Kspc(\C)$ for the trivial group and~$\Aspc$ equal to a certain $\Kspc$\nb-algebra spectrum $\Kspc X$ which may be called the \emph{incidence algebra} over~$\Kspc$ of the partially ordered set~$X$ (a finite $T_0$\nb-space is essentially the same as a partially ordered set). A construction of this form in the special case of upper-triangular $3\times 3$\nb-matrices is indicated by Schwede in~\cite{Schwede:Morita}*{\S4.5}.

The category of $\Kspc X$\nb-module spectra is Quillen equivalent to the category of diagrams of $\Kspc$\nb-module spectra indexed by~$X$. The obtained embedding
\[
\Bootstrap_\E(X)\hookrightarrow\Derived(\Kspc X)\cong\Ho\bigl(\Mod{\Kspc}^X\bigr)
\]
is interesting in its own right: it sets the stage for Morita theory, allowing us to construct equivalences $\Bootstrap_\E(X)\cong\Bootstrap_\E(Y)$ for many pairs of finite spaces~$(X,Y)$. This will enable us to treat many spaces~$X$ at once when answering questions such as: is there a manageable homology theory on $\Bootstrap_\E(X)$ computing the $\E(X)$\nb-groups via a universal coefficient theorem? These consequences will be pursued elsewhere.

\subsection*{Acknowledgement}
I owe the idea of the construction of the embedding in \S\ref{sec:ideal-related} to Ralf Meyer. I am indebted to him and Rohan Lean for valuable discussions. Further, I would like to thank Stefan Schwede and Andreas Thom for helpful correspondence. Finally, I would like to thank the anonymous referee for valuable comments on the former version of this note.

\subsection*{Some preliminaries}

We refer to~\cite{Hovey:Model_cats} as a general reference on model categories and to \cites{Schwede:Untitled_symmetric,Schwede:Homotopy_symmetric} for the theory of symmetric spectra. Recall that if~$\Model$ is a stable model category, then its homotopy category~$\Ho(\Model)$, defined as the localization of~$\Model$ at its weak equivalences, is naturally triangulated. The stable model category of module spectra over a ring spectrum~$\Rspc$ is denoted by $\Mod{\Rspc}$; its homotopy category is called the \emph{derived category} of $\Rspc$\nb-module spectra and denoted by $\Derived(\Rspc)$. We will use several times that an $\Rspc$\nb-module map is a weak equivalence if it induces isomorphisms on stable homotopy groups. 

We write $C\inOb\mathcal C$ to denote that~$C$ is an object in a category~$\mathcal C$. The \Cstar{}algebra of complex numbers is denoted by~$\C$.

\section{Connective \texorpdfstring{\E}{E}-theory}
In his thesis~\cite{Thom:Thesis}, Andreas Thom defines \emph{connective $\E$\nb-theory} for separable \Cstar{}alge\-bras. This is the universal triangulated homology theory on separable \Cstar{}algebras satisfying matrix stability (and full excision); it is denoted by~$\bu$. The category of separable \Cstar{}alge\-bras with $\bu$\nb-groups as morphisms is denoted by~$\bucat$; it carries a triangulation structure inherited from the (\Cstar{}algebra) stable homotopy category; see \cite{Thom:Thesis}*{\S3.3 and~\S4.2}.

\begin{definition}
The \emph{Spanier--Whitehead category} $\SW_\bu\subset\bucat$ in connective $\E$\nb-the\-ory is the thick triangulated subcategory generated by the \Cstar{}algebra~$\C$.
\end{definition}

It is shown in \cite{Thom:Thesis}*{Theorem~5.1.2} that there is a triangulated functor
\[
\Kspc^H\colon\bucat\to\Derived(\buspc)
\]
inducing a graded ring isomorphism $\bu_*(\C,\C)\cong\Derived(\buspc)(\buspc,\buspc)_*$, where~$\buspc\defeq\Kspc^H(\C)$ is a commutative symmetric ring spectrum equivalent to the connective $\K$\nb-theory spectrum (see \cite{Thom:Thesis}*{Propositions~5.1.1 and~D.1.1}).

\begin{proposition}
  \label{pro:fully_faithful_bu}
The functor $\Kspc^H\colon\bucat\to\Derived(\buspc)$ is fully faithful on the Spanier--Whitehead category $\SW_\bu$.
\end{proposition}

\begin{proof}
This is a standard argument; compare for instance \cite{Schwede:Morita}*{Proposition~3.10}. Consider the full subcategory of~$\bucat$ consisting of the objects~$B$ such that the map $\bu_*(\C,B)\to\Derived(\buspc)\bigl(\buspc,\Kspc^H(B)\bigr)_*$ is an isomorphism. This subcategory contains~$\C$ and is closed under suspension. It is also closed under exact triangles because~$\Kspc^H$ is triangulated. Hence it contains~$\SW_\bu$. A similar argument shows that, for fixed $B\inOb\SW_\bu$, the map $\bu_*(A,B)\to\Derived(\buspc)\bigl(\Kspc^H(A),\Kspc^H(B)\bigr)_*$ is bijective for all $A\inOb\SW_\bu$. Hence~$\Kspc^H$ is fully faithful on~$\SW_\bu$.
\end{proof}

\begin{proposition}
The essential image of the restriction of the functor $\Kspc^H$ to~$\SW_\bu$ is a triangulated subcategory of $\Derived(\buspc)$.
\end{proposition}

\begin{proof}
It suffices to prove that every morphism in the image of the restriction of~$\Kspc^H$ to~$\SW_\bu$ has a cone in the image of the restriction of~$\Kspc^H$. Such a cone can be obtained as the image of a cone of the lifting of~$f$ to~$\bucat$.
\end{proof}

Recall from \cite{Thom:Thesis}*{Theorem~5.1.2} that
\[
\pi_*\buspc
\cong\Derived(\buspc)(\buspc,\buspc)_*
\cong\bu_*(\C,\C)
\cong\Z[u],
\]
where~$u$ is of degree two. Together with the previous propositions this shows that Theorem~\ref{thm:lawson_angeltveit} may be applied to prove Theorem~\ref{thm:main} for the category~$\SW_\bu$.

\begin{remark}
We have restricted ourselves to the Spanier--Whitehead subcategory of~$\bucat$ because we do not expect~$\bucat$ to possess all countable coproducts.
\end{remark}

\section{\texorpdfstring{$G$}{G}-equivariant \texorpdfstring{\E}{E}-theory}

Let~$G$ be a compact group. A general reference for $G$\nb-equivariant $\E$\nb-theory is~\cite{Guentner-Higson-Trout:Equivariant_E}.

\begin{lemma}
  \label{lem:factorization_G}
If a functor from $\Cstarsep^G$ maps all $\E^G$\nb-equivalences to isomorphisms, then it factors through the canonical functor $\Cstarsep^G\to\Ecat^G$.
\end{lemma}

\begin{proof}
First, we observe that every element in $\E^G_0(A,B)$ can canonically be written as the composition of a $G$\nb-equivariant \Star{}ho\-mo\-mor\-phism and the inverse of another (it is straight-forward to check that the construction in \cite{Blackadar:K-theory}*{\S25.6} goes through in the $G$\nb-equivariant case). To construct the factorization, we can thus proceed as in the proof of the universal property of $\E$\nb-theory; see \cite{Blackadar:K-theory}*{Proof of 25.6.1} for details in the non-equivariant but analogous case.
\end{proof}

Now we consider the functor~$\Kspc^G\colon\Cstarsep^G\to\Mod{\Kspc^G(\C)}$ constructed in \cite{dellAmbrogio-Emerson-Kandelaki-Meyer:Functorial_K}*{\S3.3}. The construction in the non-equivariant case appeared earlier in~\cite{Joachim:K-homology_spectra}.

\begin{proposition}
The composition
\[
\Cstarsep^G\xrightarrow{\Kspc^G}\Mod{\Kspc^G(\C)}\to\Derived\bigl(\Kspc^G(\C)\bigr)
\]
descends to a triangulated functor
\[
\Kspc^G\colon\Ecat^G\to\Derived\bigl(\Kspc^G(\C)\bigr).
\]
\end{proposition}

\begin{proof}
By the previous lemma, it suffices to check that the functor $\Kspc^G\colon\Cstarsep^G\to\Mod{\Kspc^G(\C)}$ maps $\E^G$\nb-equivalences to weak equivalences. This is a consequence of the natural isomorphism
\[                                                                                                                                                                                                  
\Derived\bigl(\Kspc^G(\C)\bigr)\bigl(\Kspc^G(\C),\Kspc^G(B)\bigr)_*\cong\E^G_*(\C,B)
\]
following from~(3.6) and~(3.7) in~\cite{dellAmbrogio-Emerson-Kandelaki-Meyer:Functorial_K} and the identification $\E^G_*(\C,B)\cong\KK^G_*(\C,B)$. The fact that~$\Kspc^G$ is triangulated follows from \cite{dellAmbrogio-Emerson-Kandelaki-Meyer:Functorial_K}*{Remark~3.6} as in the proof of \cite{dellAmbrogio-Emerson-Kandelaki-Meyer:Functorial_K}*{Theorem~3.8}.
\end{proof}

\begin{definition}
The \emph{bootstrap category} (\emph{of the tensor unit})~$\Bootstrap_\E^G\subset\Ecat^G$ in $G$\nb-equi\-vari\-ant $\E$\nb-theory is the $\aleph_0$\nb-localizing subcategory generated by the \Cstar{}algebra~$\C$ with the trivial $G$\nb-action.
\end{definition}

\begin{remark}
Results in~\cite{Meyer-Nest:BC_Coactions} indicate that, if~$G$ is a (higher-dimensional) torus, then the class $\Bootstrap^G$ in fact provides the correct domain for a potential universal coefficient theorem.
\end{remark}

\begin{proposition}
  \label{pro:fully_faithful_EG}
The functor $\Kspc^G\colon\Ecat^G\to\Derived\bigl(\Kspc^G(\C)\bigr)$ is fully faithful on the bootstrap category~$\Bootstrap_\E^G$.
\end{proposition}

\begin{proof}
The proof again proceeds along the lines of \cite{Schwede:Morita}*{Proposition~3.10}. We have to check that the category~$\Ecat^G$ has countable coproducts and that the functor~$\Kspc^G$ preserves them. The former is shown in \cite{Guentner-Higson-Trout:Equivariant_E}*{Proposition~7.1}. To see the latter, we must show that the canonical map
\[
\varinjlim\limits_{n\in\N}\Kspc^G(\bigoplus\limits_{k=1}^n A_k)\to\Kspc^G(\bigoplus\limits_{k=1}^\infty A_k)
\]
is a weak equivalence for every sequence of objects $A_k\inOb\Ecat^G$. Since $G$\nb-equivariant $\K$\nb-theory preserves countable direct sums, this map induces isomorphisms on stable homotopy groups and is thus a weak equivalence.
\end{proof}

As in the previous section, the essential image of the restriction of~$\Kspc^G$ to~$\Bootstrap_\E^G$ is a triangulated subcategory of $\Derived\bigl(\Kspc^G(\C)\bigr)$.

The spectrum $\Kspc^G(\C)$ is by construction a commutative symmetric ring spectrum. In order to apply Theorem~\ref{thm:lawson_angeltveit} to prove Theorem~\ref{thm:main} in the case of~$\Bootstrap_\E^G$, we need to check that the stable homotopy groups $\pi_*\Kspc^G(\C)$ are torsion-free and concentrated in even degrees. We may identify
\[
\pi_*\Kspc^G(\C)
\cong\Derived\bigl(\Kspc^G(\C)\bigr)\bigl(\Kspc^G(\C),\Kspc^G(\C)\bigr)_*
\cong\E_*^G(\C,\C)
\cong\Repring(G)\otimes\Z[\beta,\beta^{-1}].
\]
Here $\Repring(G)$ denotes the representation ring of the group~$G$ concentrated in degree zero and~$\beta$ is an invertible element of degree~$2$ (see \cite{Blackadar:K-theory}*{Proposition~20.4.4}). Recall that the underlying Abelian group of the representation ring of~$G$ is freely generated by the isomorphism classes of simple $G$\nb-modules (see for instance~\cite{Segal:Representation_ring}). In particular, $\Repring(G)$ is torsion-free.

\section{Ideal-related \texorpdfstring{\E}{E}-theory}
  \label{sec:ideal-related}

Let~$X$ be a finite $T_0$\nb-space and let $\Cstarsep(X)$ denote the category of separable \Cstar{}al\-ge\-bras over~$X$ as defined in~\cite{MN:Bootstrap}. In particular, a \Cstar{}algebra over~$X$ is a pair $(A,\psi)$ consisting of a \Cstar{}algebra~$A$ and a continuous map from the primitive ideal space of~$A$ to~$X$. Every open subset~$U$ of~$X$ naturally gives rise to an ideal~$A(U)$ in~$A$. Let~$\Ecat(X)$ denote the version of $\E$\nb-theory for \Cstar{}algebras over~$X$ defined by Dadarlat and Meyer~\cite{Dadarlat-Meyer:E_over_space}. We refer to~$\Ecat(X)$ as \emph{ideal-related $\E$\nb-theory}.

\begin{lemma}
  \label{lem:factorization}
If a functor from $\Cstarsep(X)$ maps all $\E(X)$\nb-equivalences to isomorphisms, then it factors through the canonical functor $\Cstarsep(X)\to\Ecat(X)$.
\end{lemma}

\begin{proof}
By \cite{Dadarlat-Meyer:E_over_space}*{Lemma~2.26} (and its proof) every element in $\E_0(X;A,B)$ can canonically be written as the composition of a \Star{}ho\-mo\-mor\-phism over~$X$ and the inverse of another. To construct the factorization, we can then proceed again as in \cite{Blackadar:K-theory}*{Proof of 25.6.1}.
\end{proof}

We denote the smallest open neighbourhood of a point~$x$ in~$X$ by~$U_x$. We consider~$X$ as a partially ordered set by setting $x\leq y$ if and only if $U_x\supseteq U_y$. In order to make sense of diagrams indexed by~$X$, we regard~$X$ as a category with a unique morphism from~$x$ to~$y$ if and only if $x\geq y$. For a category~$\mathcal C$, the diagram category~$\mathcal C^X$ consists of all functors from~$X$ to~$\mathcal C$.

\begin{definition}
Let $D\colon\Cstarsep(X)\to\Cstarsep^X$ be the functor taking a \Cstar{}algebra~$A$ over~$X$ to the diagram~$D(A)$ in $\Cstarsep$ given by $D(A)(x)=A(U_x)$ and the ideal inclusions $D(A)(x\to y)=\bigl(A(U_x)\hookrightarrow A(U_y)\bigr)$. Let~$\Kspc\colon\Cstarsep\to\Mod{\Kspc(\C)}$ denote the functor of Dell'Ambrogio--Emerson--Kandelaki--Meyer with trivial group~$G$ (see \cite{dellAmbrogio-Emerson-Kandelaki-Meyer:Functorial_K}*{\S3.3}). In the following we abbreviate $\Kspc\defeq\Kspc(\C)$. Let
\[
\Kspc^X\colon\Cstarsep(X)\to\Mod{\Kspc}^X
\]
be the composition of~$D$ with pointwise application of~$\Kspc$.
\end{definition}

We equip the category $\Mod{\Kspc}^X$ with the stable model structure described in \cite{Hovey:Model_cats}*{Theorem~5.1.3}; the weak equivalences and fibrations are defined pointwise.

\begin{proposition}
The composition
\[
\Cstarsep(X)\xrightarrow{\Kspc^X}\Mod{\Kspc}^X\to\Ho\Bigl(\Mod{\Kspc}^X\Bigr)
\]
decends to a triangulated functor
\[
 \Kspc^X\colon\Ecat(X)\to\Ho\Bigl(\Mod{\Kspc}^X\Bigr).
\]
\end{proposition}

\begin{proof}
Every $\E(X)$\nb-equivalence is in particular a pointwise $\E$\nb-equivalence. Hence it is taken to a pointwise weak equivalence in $\Mod{\Kspc}^X$. By definition, every pointwise weak equivalence in $\Mod{\Kspc}^X$ is a weak equivalence and thus becomes invertible in $\Ho\Bigl(\Mod{\Kspc}^X\Bigr)$. The existence of the factorization now follows from Lemma~\ref{lem:factorization}. The fact that the induced functor~$\Kspc^X$ is triangulated is a consequence of the observations in \cite{dellAmbrogio-Emerson-Kandelaki-Meyer:Functorial_K}*{Remark~3.6}.
\end{proof}

For $x\in X$, we let $i_x\C$ denote the \Cstar{}algebra of complex numbers together with the map taking its unique primitive ideal to the point~$x\in X$. This is an object in $\Cstarsep(X)$. We set $\Repr=\bigoplus_{x\in X}i_x\C$.

Let~$\Kspc X$ denote the endomorphism ring spectrum of a stably fibrant approximation of $\Kspc^X(\Repr)$. We call this symmetric (non-commutative) $\Kspc$-algebra spectrum the \emph{indicence algebra} of~$X$ over~$\Kspc$. This construction is motivated by \cite{Schwede:Morita}*{Example~4.5(2)}. By \cite{Schwede:Morita}*{Theorem~4.16}, there is a Quillen equivalence between $\Mod{\Kspc}^X$ and $\Mod{\Kspc X}$. In particular, we will henceforth identify $\Ho\Bigl(\Mod{\Kspc}^X\Bigr)$ with $\Derived(\Kspc X)$. We have
\[
\Kspc^X(i_x\C)(y)=
\begin{cases}
\Kspc & \textup{for $y\leq x$} \\
* & \textup{else}
\end{cases}
\]
with all maps between non-trivial entries being identities. This yields a natural identification $\Hom\bigl(\Kspc^X(i_x\C),M\bigr)\cong M(x)$ for every $x\in X$ and $M\inOb\Mod{\Kspc}^X$. The corresponding relation in $\E$\nb-theory is the adjunction $\E(X;i_x\C,B)\cong\E\bigl(\C,B(U_x)\bigr)$ from \cite{Dadarlat-Meyer:E_over_space}*{(4.3)}. It follows that the graded ring homomorphism from $\E_*(X,\Repr,\Repr)$ to $\Derived(\Kspc X)\bigl(\Kspc^X(\Repr),\Kspc^X(\Repr)\bigr)_*$ induced by the functor~$\Kspc^X$ is an isomorphism and that both graded rings are isomorphic to the tensor product of the (ungraded) integral incidence algebra~$\Z X$ with the ring of Laurent polynomials $\Z[\beta,\beta^{-1}]$, where~$\beta$ has degree~$2$. The incidence algebra~$\Z X$ is the category ring of the universal pre-additive category generated by the category~$X$.

\begin{definition}
The bootstrap category $\Bootstrap_\E(X)\subset\Ecat(X)$ is the $\aleph_0$\nb-localizing subcategory generated by the object~$\Repr$.
\end{definition}

\begin{proposition}
The functor $\Kspc^X\colon\Ecat(X)\to\Derived(\Kspc X)$ is fully faithful on the bootstrap category~$\Bootstrap_\E(X)$.
\end{proposition}

\begin{proof}
The proof is essentially analogous to the one of Proposition~\ref{pro:fully_faithful_EG}. We use the fact that ideal-related $\K$\nb-theory preserves countable inductive limits.
\end{proof}

As before, the essential image of the restriction of~$\Kspc^X$ to~$\Bootstrap_\E(X)$ is a triangulated subcategory of $\Derived(\Kspc X)$ and the computation
\[
\pi_*\Kspc X
\cong\Derived(\Kspc X)\bigl(\Kspc^X(\Repr),\Kspc^X(\Repr)\bigr)_*
\cong\E_*(X;\Repr,\Repr)
\cong\Z X\otimes\Z[\beta,\beta^{-1}]
\]
with $\Z X$ concentrated in degree zero and~$\beta$ of degree~$2$ shows that we can apply Theorem~\ref{thm:lawson_angeltveit} to prove Theorem~\ref{thm:main} in the case of~$\Bootstrap_\E(X)$.

\section{Conclusion}

We have shown that certain triangulated categories related to \Cstar{}algebras have infinite $n$\nb-order for every $n\in\N$ by relating them with certain ring spectra. This means that they share many structural properties of algebraic triangulated categories, but it is not clear whether they are actually algebraic. To conclude, we mention that in the specific case of the derived category of $\Kspc$\nb-module spectra there is an equivalence to an algebraic triangulated category, but it is not known whether this equivalence is triangulated; see \cite{Patchkoria:Alg_Class}*{Theorem~3.1.4.(ii)}.

\end{document}